\documentclass{article}
\usepackage{graphicx,xcolor} 

\usepackage{amsmath}
\allowdisplaybreaks
\usepackage{amssymb,amsthm,arydshln,bm, enumitem}

\theoremstyle{plain}
\newtheorem{thm}{Theorem}
\newtheorem{prop}[thm]{Proposition}
\newtheorem{lem}[thm]{Lemma}
\newtheorem{cor}[thm]{Corollary}

\theoremstyle{definition}
\newtheorem{dfn}[thm]{Definition}

\newtheorem{rem}[thm]{Remark}
\newtheorem{note}[thm]{Note}

\title{\bf
A Constructive Approach to Infinitesimal Conformal Rigidity on Complex Hyperbolic Space}

\author{Hiroyasu Satoh and Hemangi Madhusudan Shah}
\date{May 2025}

\begin{document}

\maketitle

\begin{abstract}

We prove that every conformal vector field on the complex hyperbolic space $\mathbb{C}H^n$ is Killing for all $n\ge 2$.
Although this rigidity is classically known, our proof is entirely different in nature: it is local, analytic, and fully constructive.
Our approach is local, analytic, and constructive: we view $\mathbb{C}H^2$ through its solvable Lie group model and express the conformal Killing equation as an explicit system of partial differential equations.
By solving this system completely, we show that any conformal vector field must be determined by a Killing field.
The analysis in complex dimension $2$ naturally extends to arbitrary $n$, yielding a unified and fully explicit proof of this rigidity phenomenon.
\end{abstract}

\begin{itemize}[leftmargin=73pt]
    \item[{\bf Keywords:}] Complex hyperbolic space; 
Conformal vector field;\\  Killing vector field; Damek-Ricci space
\end{itemize}



\section{Introduction}

The study of conformal vector fields is a classical and fundamental topic in Riemannian geometry. It is well known that the standard space forms $\mathbb{R}^n$, $\mathbb{S}^n$, and $\mathbb{R}H^n$ are distinguished by the richness of their conformal vector fields.
In particular, Tashiro \cite{Tashiro1965} showed that the existence of certain conformal or concircular vector fields imposes strong geometric restrictions on the underlying manifold, and in some cases, completely characterizes the global geometry.
Lichnerowicz \cite{L, L-F} conjectured that every complete Riemannian manifold whose conformal group is \emph{essential} must be conformally equivalent to either the Euclidean space or the standard sphere.
For complete Einstein manifolds, the existence of a conformal vector field satisfying an appropriate non-triviality condition (see the two cases treated in \cite{YanoNagano51}) forces the metric to be isometric to either the Euclidean space or the standard sphere.
The general case (with neither Einstein nor compactness assumptions) was later established by Lelong–Ferrand, completing and correcting an earlier approach of Alekseevski\u{i}; see \cite{L-F, Alekseevskii72} for a concise historical overview.
The nonexistence result we establish in this paper (see Theorem~\ref{thm:main} below) 
provides a direct, local PDE proof of a related rigidity phenomenon, obtained without appealing to global splitting arguments or group actions.

In this article, we focus on the complex hyperbolic space $\mathbb{C}H^n(-1)$, which is a noncompact, rank-one symmetric space and also a special example of a Damek-Ricci space of constant holomorphic sectional curvature $-1$.
In the sequel, $\mathbb{C}H^n$ would mean
$\mathbb{C}H^n(-1)$.
We prove that all conformal vector fields on $\mathbb{C}H^n$ are Killing, thus exhibiting a rigidity phenomenon in contrast to the behavior seen in other space forms.

\begin{thm}\label{thm:main}
Let $\xi$ be a conformal vector field on the complex hyperbolic space $\mathbb{C}H^n$ for $n \geq 2$. Then $\xi$ is a Killing vector field.
\end{thm}

After completing the first version of this work, we came across Kanai’s theorem \cite[Theorem G]{Ka}, which refines an earlier result of Tashiro and states that an Einstein manifold admitting a non-homothetic conformal vector field must be one of the real space forms.
Hence, our result can be regarded as a special case of this theorem.
Furthermore, the same conclusion can also be derived from Nickerson’s observation \cite[Remark~2]{Nickerson1985}:
if a locally symmetric and non-conformally flat Riemannian manifold of dimension $\ge 4$ with non-zero (hence constant) scalar curvature admits a conformal vector field, then it must be Killing.
We formulate the corresponding statement in our local setting and provide a brief proof in Appendix.
Hence our theorem may also be viewed as a local re-derivation of this classical rigidity phenomenon.
However, our proof is entirely different in nature.
We regard $\mathbb{C}H^n$ as a Damek–Ricci space and express the conformal Killing condition as an explicit system of partial differential equations on the solvable model.
This analytic viewpoint yields a direct and constructive proof: it not only establishes the nonexistence of non-Killing conformal fields, but also determines the form of all Killing fields.
Furthermore, the same analytic method extends naturally to general Damek–Ricci spaces, which constitute the only known examples of non-symmetric harmonic spaces.
In this sense, the present work may be regarded as a first step toward a broader understanding of infinitesimal conformal rigidity in the family of Damek–Ricci spaces, to be developed further in our forthcoming paper \cite{SS2}.
This demonstrates a strong rigidity property of $\mathbb{C}H^n$: although the space is non-compact and highly symmetric, its conformal group coincides with its isometry group.
Unlike classical global arguments relying on Bochner–Yau integral formulas or group-theoretic considerations, our proof is local and analytic, treating the conformal Killing system directly as a PDE on the solvable model.



\begin{rem}
A related rigidity phenomenon is known in the compact setting.
A classical theorem of Lichnerowicz (see \cite[2.125 Corollary]{Besse}) asserts that any conformal vector field on a compact Kähler manifold of real dimension at least four must be Killing.
Our result may thus be regarded as a noncompact analogue of this rigidity, now in the context of Kähler–Einstein manifolds.
\end{rem}

This paper is organized as follows.  
In Section 2, we recall the definition and basic properties of conformal vector fields and present an explicit construction of nontrivial conformal vector fields on real hyperbolic space as a motivating example.
In Section 3, we consider the complex hyperbolic plane $\mathbb{C}H^2$ viewed as a Damek–Ricci space, and derive the system of partial differential equations that a conformal vector field must satisfy.
In Section 4, we solve this system and show that all conformal vector fields on $\mathbb{C}H^2$ are necessarily Killing.
In Section 5, we extend this analysis to the general case of complex hyperbolic space $\mathbb{C}H^n$, completing the proof of our main result.
Finally, in the Appendix, we include a brief proof showing how Theorem~\ref{Nickerson} follows from Nickerson’s Remark~2, for the reader’s convenience.

\section{Conformal vector fields}

In this section we describe conformal vector fields 
on a Riemannian manifold and explain the construction 
of conformal vector fields on $\mathbb{R}H^n$. 
We also explain the geometry behind the fact that why nontrivial conformal vector fields do not exist on 
$\mathbb{C}H^n$.\\

Let $(M, g)$ be a complete Riemannian manifold and
$\mathfrak{X}(M)$ be the Lie algebra of smooth vector ﬁelds on $M$.

\begin{dfn}
A smooth vector field $\xi$ is said to be a {\it conformal} vector field, if there exists $\rho\in C^\infty(M)$ that satisfies $\mathcal{L}_{\xi}g=2\rho g$, where $\mathcal{L}_{\xi}g$ is the Lie derivative of $g$ with respect to $\xi\in \mathfrak{X}(M)$.
We call $\rho$ the potential function of $\xi$.
In particular, if the potential function $\rho$ is constant,
then $\xi$ is called a {\it homothetic} vector field.
If $\rho=0$, then $\xi$ is called a {\it Killing} vector field.
\end{dfn}

Conformal vector fields are infinitesimal generators of local conformal transformations of the manifold $(M, g)$, that is, smooth one-parameter families of diffeomorphisms $\{\varphi_t\}$ such that $\varphi_t^*g = e^{2\rho_t}g$ for some smooth functions $\rho_t$ depending on $t$.
This classical correspondence between conformal vector fields and conformal transformations has been studied in detail in the foundational works of Yano \cite[Chapter VII]{Yano1957}.

\begin{lem}\label{property_L}
The Lie derivative of the metric $g$ satisfies 
\begin{enumerate}
\item $\mathcal{L}_{\xi_1+\xi_2}g
=\mathcal{L}_{\xi_1}g+\mathcal{L}_{\xi_2}g$\quad$(\xi_1, \xi_2\in \mathfrak{X}(M)),$
\item $\mathcal{L}_{f\xi}g
=f\mathcal{L}_{\xi}g+2\,\mathrm{sym}(df\otimes \xi^\flat)$\quad
$(f\in C^\infty(M),\ \xi\in \mathfrak{X}(M)).$
\end{enumerate}
Here $\xi^\flat=g(\xi,\,\cdot\,)$ is the dual 1-form of $\xi$ with respect to $g$ and $\mathrm{sym}(df\otimes \xi^\flat)$ is the symmetric tensor product of $df$ and $\xi^\flat$;
$$
\mathrm{sym}(df\otimes \xi^\flat)=
\dfrac{1}{2}(df\otimes \xi^\flat+\xi^\flat\otimes df).
$$
\end{lem}

\begin{proof}
We prove only (2), since (1) is a direct consequence of the linearity of the Lie derivative.
Let $X, Y \in \mathfrak{X}(M)$ be arbitrary vector fields.
Then,
\begin{align*}
(\mathcal{L}_{f\xi}g)(X, Y)
=&g(\nabla_X(f\xi), Y)+g(X, \nabla_Y(f\xi)),\\
=&g(X(f)\,\xi+f\nabla_X \xi, Y)+g(X, Y(f)\,\xi+f\nabla_Y \xi),\\
=&f\left(
g(\nabla_X \xi, Y)+g(X, \nabla_Y \xi)
\right)
+X(f)\,g(\xi, Y)+Y(f)\,g(X, \xi),\\
=&f\,\mathcal{L}_{\xi}g(X, Y)+2\,\mathrm{sym}(df\otimes \xi^\flat)(X, Y),
\end{align*}
as claimed.
\end{proof}

From this lemma, it follows that the set of all conformal vector fields is a real vector space.
Its dimension coincides with that of the local conformal transformation group. It is a classical result that this dimension is bounded above by $\frac{1}{2}(n+1)(n+2)$ for an $n$-dimensional Riemannian manifold with $n \geq 3$ (see \cite[p.310, Theorem 1]{KN1}).
Moreover, if a vector field $\xi$ is conformal with respect to a metric $g$ with potential function $\rho$,
then $\xi$ is also a conformal vector field with respect to the conformally related metric $\tilde{g} = e^{2f}g$,
with a new potential function given by $\tilde{\rho} = \rho + \xi(f)$.

As an illustrative example, we now describe how conformal vector fields can be constructed on the real hyperbolic space $\mathbb{R}H^n$ using Busemann functions.
This construction exploits special properties of Busemann functions that are specific to the real hyperbolic geometry, and does not extend to the complex hyperbolic case discussed later.
Let us first recall the definition of the Busemann function.

\begin{dfn}
Let $(M,g)$ be a complete non-compact Riemannian manifold,
and let $v \in T_pM$ be a unit tangent vector at a point $p \in M$.
The \textit{Busemann function} associated with $v$ is the function $b_v : M \to \mathbb{R}$ defined by
\[
b_v(x) = \lim_{t \to \infty}\left(d(x, \gamma_v(t)) - t\right),
\]
where $\gamma_v : [0, \infty) \to M$ is the geodesic ray with $\gamma_v'(0)=v$ and $d(\cdot,\cdot)$ denotes the distance function on $(M,g)$.
\end{dfn}

\begin{prop}\label{hess b}
Assume that there exists a function $b\in C^\infty(M)$ satisfying that
\begin{equation}\label{cond_Hessian}
\mathrm{Hess}(b)=g-db\otimes db,
\end{equation}
where $\mathrm{Hess}(b)$ is the Hessian of $b$.
Then, $\xi=\nabla \rho$, where $\rho=\exp(b)$, is a conformal vector field with the potential function $\rho$.
\end{prop}

\begin{proof}
From Lemma \ref{property_L} (2), we have
\begin{align*}
\mathcal{L}_{\xi}g
=\mathcal{L}_{\rho \nabla b}g
=&\rho \mathcal{L}_{\nabla b}g+\mathrm{sym}(d\rho\otimes (\nabla b)^\flat)\\
=&\rho\cdot 2\mathrm{Hess}(b)
+\mathrm{sym}(\rho\,db\otimes db)\\
=&\rho \left(
2\mathrm{Hess}(b)+2db\otimes db
\right)
=2\rho\,g.
\end{align*}
\end{proof}

In the real hyperbolic space of constant sectional curvature $-1$, it is known that for any unit tangent vector  $v$,
the associated Busemann function $b_v$ satisfies equation \eqref{cond_Hessian} (see \cite[p.751]{BCG}).
In fact, this identity characterizes the constant curvature condition, that is, \eqref{cond_Hessian} holds if and only if the ambient space has constant sectional curvature $-1$.
Therefore, the following observation immediately follows.

\begin{cor}
Let $\mathbb{R}H^n(-1)$ be the real hyperbolic space of constant sectional curvature $(-1)$.
Then, for any unit tangent vector $v$,
$\xi=\nabla \rho$, where $\rho=\exp(b_v)$ is a conformal vector field on $\mathbb{R}H^n(-1)$.
\end{cor}




It is important to note that Proposition~\ref{hess b} does not imply that conformal vector fields exist only on real hyperbolic spaces of sectional curvature $-1$.
Since the metric $g_\lambda:=\frac{1}{\sqrt{-\lambda}}g_{-1}$ defines the real hyperbolic space $\mathbb{R}H^n(\lambda)$ of constant sectional curvature $\lambda < 0$, and conformal vector fields remain conformal under conformal changes of the metric, the same vector fields constructed using the Busemann function for $\mathbb{R}H^n(-1)$ also define conformal vector fields on $\mathbb{R}H^n(\lambda)$.
More generally, since a conformal vector field with respect to a metric $g$ remains conformal for any conformally related metric $e^{2f}g$, this construction applies to all spaces conformally equivalent to $\mathbb{R}H^n(-1)$, using the same underlying Busemann function.


However, this method does not extend directly to other noncompact rank-one symmetric spaces.
In these spaces, the Hessian of the Busemann function has two distinct eigenvalues (see \cite[p.751]{BCG}).
This anisotropy prevents the direct construction of conformal vector fields via the Busemann function, as in the real hyperbolic case.
In the next section, we turn to the case of $\mathbb{C}H^n$ and examine how conformal vector fields may or may not arise in that setting.

\begin{rem}
It is well known that the isometry group of the real hyperbolic space $\mathbb{R}H^n$ has dimension $\frac{1}{2}n(n+1)$, and that the maximum dimension of the conformal transformation group on an $n$-dimensional Riemannian manifold is $\frac{1}{2}(n+1)(n+2)$.
Busemann functions, which are initially defined via unit tangent vectors $v \in T_p\mathbb{R}H^n$, can equivalently be viewed as depending on a base point $p \in M$ and a point at infinity $\partial M$.
In the upper half-space model of $\mathbb{R}H^n$, fixing the base point at $(\bm{0},1)$, we can express the Busemann functions associated to the ideal point at infinity $\infty$ and a boundary point $(\bm{a}, 0)$ as follows:
$$
b_\infty(\bm{x}, y) = - \log y, \quad
b_{\bm{a}}(\bm{x}, y) = \log \left( \frac{|\bm{x} - \bm{a}|^2+y^2}{y} \right).
$$
By computing $\nabla \exp(b)$ for each of these functions, it follows immediately that the conformal vector fields generated by Busemann functions span an $(n+1)$-dimensional space of non-Killing conformal vector fields.
In particular, this shows that all non-Killing conformal vector fields arise from Busemann functions, and that there are no nontrivial homothetic vector fields on $\mathbb{R}H^n$.
\end{rem}

\section{Conformal vector fields on $\mathbb{C}H^2$}
\label{CH2}

In this section, first we explain the system of partial differential equations leading to a conformal vector field equation on $\mathbb{C}H^2$.\\
A complex hyperbolic plane $\mathbb{C}H^2$ is the Damek-Ricci space $S=N\ltimes \mathbb{R}A$, where the Lie algebra $\mathfrak{n} = \mathfrak{v}\oplus \mathfrak{z}$ of $N$ has a one-dimensional center $\mathfrak{z}=\mathbb{R}Z$ and $\mathfrak{v}=\mathrm{span}\{ V, J_Z V\}$,
where $V\in \mathfrak{v}$ and $Z\in\mathfrak{z}$ are unit vectors and $J_Z : \mathfrak{v}\rightarrow \mathfrak{v}$ a linear map defined by $\langle J_Z V, V'\rangle=\langle Z, [V, V']\rangle$ (see \cite[pp.22]{BTV}).
Any vector field on $S=\mathbb{C}H^2$ can be expressed using the left-invariant vector fields as
\begin{equation}\label{xi}
\xi=f_1\,V+f_2\,J_ZV+f_3\,Z+f_4\,A,\qquad f_i\in C^\infty(S).
\end{equation}
In this case, we examine whether there exist functions $f_1, f_2, f_3$, and $f_4$ such that $\xi$ becomes a conformal vector field.

\subsection{The Lie derivative components on $\mathbb{C}H^2$}

The Lie bracket on $\mathfrak{s} = \mathrm{Lie}(S)$ is generally defined as follows \cite[pp.79]{BTV}:
\begin{equation*}\label{DRbp}
[V_1+Z_1+a_1 A, V_2+Z_2+a_2 A]_{\mathfrak{s}} := [V_1, V_2]_{\mathfrak{n}} + \frac{a_1}{2}V_2 - \frac{a_2}{2}V_1 + a_1 Z_2 - a_2 Z_1,
\end{equation*}
where $V_i\in \mathfrak{v}, Z_i \in \mathfrak{z}, a_i\in \mathbb{R}$.
In particular, in $\mathbb{C}H^2$,
we obtain
\begin{align*}
[V, J_ZV]=&Z, &[V,Z]=&[J_ZV, Z]=0, &&\\
[A,V]=&\dfrac{1}{2}V, &[A, J_Z V]=&\dfrac{1}{2}J_ZV, &[A, Z]=&Z.
\end{align*}
Considering the elements of the Lie algebra as left-invariant vector fields on $\mathbb{C}H^2$,
the components of the Lie derivative of the metric $g$ are as follows:
\begin{equation*}
\mathcal{L}_Vg(V,A)=\frac{1}{2},\quad
\mathcal{L}_Vg(J_ZV,Z)=-1,
\end{equation*}
\begin{equation*}
\mathcal{L}_{J_ZV}g(V,Z)=1,\quad
\mathcal{L}_{J_ZV}g(J_ZV,A)=\dfrac{1}{2},
\end{equation*}
\begin{equation*}
\mathcal{L}_{Z}g(Z,A)=1,
\end{equation*}
\begin{equation*}
\mathcal{L}_{A}g(V,V)=\mathcal{L}_{A}g(J_ZV,J_ZV)=-1,\quad
\mathcal{L}_{A}g(Z,Z)=-2.
\end{equation*}
Here, all components not listed above are zero.
Hence, we have
\begin{equation}\label{Lie_g}
\begin{split}
\mathcal{L}_Vg
=\left(
\begin{array}{cc:cc}
{ }&{ }&\ast&\ast\\
{ }&{ }&\ast&\ast\\
\hdashline
0&-1&&\\
\frac{1}{2}&0&&
\end{array}
\right),&\qquad
\mathcal{L}_{J_ZV}g
=\left(
\begin{array}{cc:cc}
{ }&{ }&\ast&\ast\\
{ }&{ }&\ast&\ast\\
\hdashline
1&0&&\\
0&\frac{1}{2}&&
\end{array}
\right),\\
\mathcal{L}_Zg
=\left(
\begin{array}{cc:cc}
\hspace{15pt}&{ }&{ }&{ }\\
{ }&{ }&{ }&{ }\\
\hdashline
{ }&{ }&0&1\\
{ }&{ }&1&0
\end{array}
\right),&\qquad
\mathcal{L}_Ag
=\left(
\begin{array}{cc:cc}
-1&0&&\\
0&-1&&\\
\hdashline
&&-2&0\\
&&0&0
\end{array}
\right).
\end{split}
\end{equation}
Here, the blank blocks represent zero matrices.





\subsection{Exterior derivative of a function on $\mathbb{C}H^2$}

In this subsection, we express the 1-form $df$ for $f\in C^\infty(S)$ using the dual forms of the left-invariant vector fields $V, J_Z V, Z, A$.
For this purpose, we introduce a global coordinate system on the Damek-Ricci space.
Let $\{e_i\}_{i=1,2,\ldots,k}$ and $\{e_{k+r}\}_{r=1,2,\ldots,m}$ be orthonormal bases for $\mathfrak{v}$ and $\mathfrak{z}$, respectively.
Through the exponential map,
we can identify $S$ with $\mathbb{R}^{k+m+1}$ and we
coordinatize $\mathbb{R}^{k+m+1}$ as
$(v_1, \ldots, v_k, z_1, \ldots, z_m, a)$.
Using this coordinate system, the left-invariant vector fields can be expressed as
\begin{align*}
E_i=&e^{a/2}\dfrac{\partial}{\partial v_i}
-\dfrac{e^{a/2}}{2}\sum_{j}\sum_{r}A^r_{ij}v_j\dfrac{\partial}{\partial z_r},\\
E_{k+r}=&e^{a}\dfrac{\partial}{\partial z_r},\\
E_0=&\dfrac{\partial}{\partial a},
\end{align*}
where $A^r_{ij}=g([e_i, e_j], e_{k+r})$ \cite[pp.82]{BTV}.
When $S=\mathbb{C}H^2$,
setting $e_1 = V, e_2 = J_Z V$, and $e_3 = Z$, and $(v_1, v_2, z, a)=(x, y, z, a)$,
we obtain
$$
V=e^{a/2}\left(\dfrac{\partial}{\partial x}
-\dfrac{y}{2}\dfrac{\partial}{\partial z}\right),\ 
J_ZV=e^{a/2}\left(\dfrac{\partial}{\partial y}
+\dfrac{x}{2}\dfrac{\partial}{\partial z}\right),\ 
Z=e^{a}\dfrac{\partial}{\partial z},\ 
A=\dfrac{\partial}{\partial a}.
$$
Hence, the 1-form $df$ is expressed as the form
\begin{align*}
df
=&Vf\cdot V^\flat + (J_ZV)f\cdot (J_ZV)^\flat + Zf\cdot Z^\flat + Af\cdot A^\flat\\
=&
e^{a/2}\left(\dfrac{\partial f}{\partial x}-\dfrac{y}{2}\dfrac{\partial f}{\partial z}\right)
V^\flat
+
e^{a/2}\left(\dfrac{\partial f}{\partial y}+\dfrac{x}{2}\dfrac{\partial f}{\partial z}\right)
(J_ZV)^\flat
+
e^{a} \dfrac{\partial f}{\partial z}
Z^\flat
+
\dfrac{\partial f}{\partial a}
A^\flat.
\end{align*}

From the above, we obtain, 
\begin{equation}\label{dfxV}
2\mathrm{sym}(df_1\otimes V^\flat+df_2\otimes (J_Z V)^\flat+df_3\otimes Z^\flat+df_4\otimes A^\flat)\\
=\left(
\begin{array}{cc}
\mathcal{F}_{11}&{ }^t\mathcal{F}_{21}\\
\mathcal{F}_{21}&\mathcal{F}_{22}
\end{array}
\right),
\end{equation}
where
\begin{align*}
\mathcal{F}_{11}
=&e^{a/2}\left(
\begin{array}{cc}
2\left(\frac{\partial f_1}{\partial x}-\frac{y}{2}\frac{\partial f_1}{\partial z}\right)
&
\ast\\
\left(\frac{\partial f_1}{\partial y}+\frac{\partial f_1}{\partial z}\right)
+
\left(\frac{\partial f_2}{\partial x}-\frac{y}{2}\frac{\partial f_2}{\partial z}\right)
&
2\left(\frac{\partial f_2}{\partial y}+\frac{x}{2}\frac{\partial f_2}{\partial z}\right)
\end{array}
\right),\\
\mathcal{F}_{21}
=&\left(
\begin{array}{cc}
e^{a} \frac{\partial f_1}{\partial z}
+
e^{a/2}\left(\frac{\partial f_3}{\partial x}-\frac{y}{2}\frac{\partial f_3}{\partial z}\right)
&
e^{a} \frac{\partial f_2}{\partial z}
+
e^{a/2}\left(\frac{\partial f_3}{\partial y}+\frac{x}{2}\frac{\partial f_3}{\partial z}\right)\\
\frac{\partial f_1}{\partial a}
+
e^{a/2}\left(\frac{\partial f_4}{\partial x}-\frac{y}{2}\frac{\partial f_4}{\partial z}\right)
&
\frac{\partial f_2}{\partial a}
+
e^{a/2}\left(\frac{\partial f_4}{\partial y}+\frac{x}{2}\frac{\partial f_4}{\partial z}\right)
\end{array}
\right),\\
\mathcal{F}_{22}
=&\left(
\begin{array}{cc}
2e^{a} \frac{\partial f_3}{\partial z}
&
\ast
\\
\frac{\partial f_3}{\partial a}
+
e^{a} \frac{\partial f_4}{\partial z}
&
2\frac{\partial f_4}{\partial a}
\end{array}
\right).
\end{align*}

\subsection{Existence of conformal vector fields on $\mathbb{C}H^2$ and related results}

In this subsection, we establish the system of differential equations corresponding to  the equation
$\mathcal{L}_\xi g = 2 fg.$ We also show that $\xi$ is a Killing vector field if (i) it is a homothetic vector field or (ii) if $e^a f_3$ or $e^a f_4$ is a radial harmonic function in the $zw$-plane.
Furthermore, we show that if $\xi$ is a Killing vector
field than it can neither be closed nor it can have constant length.\\

For a vector field $\xi$ written as \eqref{xi}, from Lemma \ref{property_L},
we have
\begin{align*}
\mathcal{L}_\xi g
=&f_1 \mathcal{L}_V g+f_2 \mathcal{L}_{J_Z V} g+f_3 \mathcal{L}_Z g+f_4 \mathcal{L}_A g\\
&+2\mathrm{sym}(df_1\otimes V^\flat+df_2\otimes (J_Z V)^\flat+df_3\otimes Z^\flat+df_4\otimes A^\flat).
\end{align*}
From equations \eqref{Lie_g} and \eqref{dfxV},
the condition that $\xi$ satisfies $\mathcal{L}_\xi g=2\rho g$ is equivalent to the following equations:
\begin{align}
\label{(1,1)}
2e^{a/2}\left(\frac{\partial f_1}{\partial x}-\frac{y}{2}\frac{\partial f_1}{\partial z}\right)
-f_4
=&2\rho,\\
\label{(2,1)}
e^{a/2}\left(\frac{\partial f_1}{\partial y}+\frac{x}{2}\frac{\partial f_1}{\partial z}\right)
+
e^{a/2}\left(\frac{\partial f_2}{\partial x}-\frac{y}{2}\frac{\partial f_2}{\partial z}\right)
=&0,\\
\label{(2,2)}
2e^{a/2}\left(\frac{\partial f_2}{\partial y}+\frac{x}{2}\frac{\partial f_2}{\partial z}\right)
-f_4
=&2\rho,\\
e^{a} \frac{\partial f_1}{\partial z}
+
e^{a/2}\left(\frac{\partial f_3}{\partial x}-\frac{y}{2}\frac{\partial f_3}{\partial z}\right)
+f_2
=&0,\\
e^{a} \frac{\partial f_2}{\partial z}
+
e^{a/2}\left(\frac{\partial f_3}{\partial y}+\frac{x}{2}\frac{\partial f_3}{\partial z}\right)
-f_1
=&0,\\
\label{(3,3)}
2e^{a} \frac{\partial f_3}{\partial z}
-2f_4
=&2\rho,\\
\label{(4,1)}
\frac{\partial f_1}{\partial a}
+
e^{a/2}\left(\frac{\partial f_4}{\partial x}-\frac{y}{2}\frac{\partial f_4}{\partial z}\right)
+\dfrac{1}{2}f_1
=&0,\\
\label{(4,2)}
\frac{\partial f_2}{\partial a}
+
e^{a/2}\left(\frac{\partial f_4}{\partial y}+\frac{x}{2}\frac{\partial f_4}{\partial z}\right)
+\dfrac{1}{2}f_2
=&0,\\
\label{(4,3)}
\frac{\partial f_3}{\partial a}
+
e^{a} \frac{\partial f_4}{\partial z}
+f_3
=&0,\\
\label{(4,4)}
2\frac{\partial f_4}{\partial a}
=&2\rho.
\end{align}
By substituting \eqref{(4,4)} into \eqref{(1,1)}, \eqref{(2,2)}, and \eqref{(3,3)}, and simplifying, we obtain
\begin{align}
\label{(1,1)'}
\left(\frac{\partial f_1}{\partial x}-\frac{y}{2}\frac{\partial f_1}{\partial z}\right)
=&e^{-a}\dfrac{\partial}{\partial a}\left(e^{a/2}f_4\right),\\
\label{(2,1)'}
\left(\frac{\partial f_1}{\partial y}+\frac{x}{2}\frac{\partial f_1}{\partial z}\right)
+
\left(\frac{\partial f_2}{\partial x}-\frac{y}{2}\frac{\partial f_2}{\partial z}\right)
=&0,\\
\label{(2,2)'}
\left(\frac{\partial f_2}{\partial y}+\frac{x}{2}\frac{\partial f_2}{\partial z}\right)
=&e^{-a}\dfrac{\partial}{\partial a}\left(e^{a/2}f_4\right),\\
\label{(3,1)}
e^{a} \frac{\partial f_1}{\partial z}
+
e^{a/2}\left(\frac{\partial f_3}{\partial x}-\frac{y}{2}\frac{\partial f_3}{\partial z}\right)
+f_2
=&0,\\
\label{(3,2)}
e^{a} \frac{\partial f_2}{\partial z}
+
e^{a/2}\left(\frac{\partial f_3}{\partial y}+\frac{x}{2}\frac{\partial f_3}{\partial z}\right)
-f_1
=&0,\\
\label{(3,3)'}
\dfrac{\partial}{\partial z}\left(e^af_3\right)
=&e^{-a}\dfrac{\partial}{\partial a}\left(e^af_4\right),\\
\label{(4,1)'}
e^a\left(\frac{\partial f_4}{\partial x}-\frac{y}{2}\frac{\partial f_4}{\partial z}\right)
+\dfrac{\partial}{\partial a}\left(e^{a/2}f_1\right)
=&0,\\
\label{(4,2)'}
e^a\left(\frac{\partial f_4}{\partial y}+\frac{x}{2}\frac{\partial f_4}{\partial z}\right)
+\dfrac{\partial}{\partial a}\left(e^{a/2}f_2\right)
=&0,\\
\label{(4,3)'}
\dfrac{\partial}{\partial z}\left(e^af_4\right)
=&-e^{-a}\dfrac{\partial}{\partial a}\left(e^af_3\right).
\end{align}

Thus, the existence of conformal vector fields on $\mathbb{C}H^2$ is reduced to the existence of non-trivial functions $f_i : \mathbb{R}^4 \to \mathbb{R}$ that satisfy the partial differential equations \eqref{(1,1)'} through \eqref{(4,3)'}.\\


Recall that a conformal vector field is said to be {\it homothetic} if the conformal factor is constant.
From these observations, we can first deduce the following fundamental property.

\begin{thm}\label{homo}
If $\xi$ is a homothetic vector field on $\mathbb{C}H^2$, then it is a Killing vector field. 
\end{thm}

\begin{proof}
The proof immediately follows from the above equations.
We suppose that $\rho\equiv C_0$.
From \eqref{(4,4)} we have $\dfrac{\partial f_4}{\partial a}=C_0$, i.e.,
$f_4$ can be expressed using a function $C_1(x,y,z)$ as
\begin{equation*}
f_4(x,y,z,a)=C_0\,a+C_1(x,y,z).
\end{equation*}
From \eqref{(3,3)'}, we have
\begin{align*}
\frac{\partial}{\partial z}\left(e^a f_3\right)
=e^{-a}\frac{\partial}{\partial a}\left(
e^a f_4
\right)
=&e^{-a}\frac{\partial}{\partial a}\left\{
e^a(C_0 a+C_1(x,y,z))
\right\}\\
=&e^{-a}\left\{
e^a(C_0 a+C_1(x,y,z))+e^a C_0
\right\}\\
=&C_0(a+1)+C_1(x,y,z).
\end{align*}
Hence, we can write
\begin{equation*}
e^a f_3(x,y,z,a)
=C_0(a+1) z + C_2(x,y,z)+C_3(x,y,a),
\end{equation*}
where $C_2(x,y,z)$ is a function satisfying $\frac{\partial C_2}{\partial z}=C_1$ and $C_3(x,y,a)$ is a function that does not depend on the variable $z$.
The left-hand side of \eqref{(4,3)'} becomes
\begin{equation*}
\dfrac{\partial}{\partial z}\left(e^af_4\right)
=e^a\dfrac{\partial}{\partial z}\left\{C_0 a+C_1(x,y,z)\right\}
=e^a\dfrac{\partial C_1}{\partial z}(x,y,z).
\end{equation*}
On the other hand, the right-hand side becomes
\begin{align*}
-e^{-a}\dfrac{\partial}{\partial a}\left(e^af_3\right)
=&-e^{-a}\dfrac{\partial}{\partial a}\left\{
C_0(a+1) z + C_2(x,y,z)+C_3(x,y,a)
\right\}\\
=&-e^{-a}\left(
C_0 z+\frac{\partial C_3}{\partial a}(x,y,a)\right),
\end{align*}
so the following equation holds:
\begin{equation}\label{eq_C1}
C_1(x,y,z)
=-e^{-2a}\left\{
\dfrac{C_0}{2}z^2
+\frac{\partial C_3}{\partial a}(x,y,a)\,z
\right\}
+C_4(x,y,a),
\end{equation}
where $C_4(x,y,a)$ is a function that does not depend on the variable $z$.
In \eqref{eq_C1}, setting $z = 0$ gives
\begin{equation*}
C_1(x,y,0)=C_4(x,y,a).
\end{equation*}
Hence, $C_1$ can be written as
\begin{equation}\label{eq_C1’}
C_1(x,y,z)
=-e^{-2a}\left\{
\dfrac{C_0}{2}z^2
+\frac{\partial C_3}{\partial a}(x,y,a)\,z\right\}
+C_1(x,y,0).
\end{equation}
Differentiating both sides of \eqref{eq_C1’} with respect to $z$ and substituting $z = 0$ yields
\begin{equation*}
\frac{\partial C_1}{\partial z}(x,y,0)=-e^{-2a}\frac{\partial C_3}{\partial a}(x,y,a).
\end{equation*}
Hence, $C_1$ can be written as
\begin{equation}\label{eq_C1''}
C_1(x,y,z)
=-\dfrac{C_0}{2} e^{-2a} z^2
+\frac{\partial C_1}{\partial z}(x,y,0)\,z
+C_1(x,y,0).
\end{equation}
Since $C_1(x,y,z)$ does not depend on the variable $a$, 
$C_0$ must be zero.
\end{proof}

\begin{rem}
\begin{itemize}
\item[(1)] From the above theorem, we recover the result that the Busemann function can't be a homothetic vector field, as it is not a Killing vector field on $\mathbb{C}H^n$.
\item[(2)] From Corollary 2.2 of \cite{KR}, we see that on any Einstein manifold a homothetic vector field is Killing.
Thus Theorem \ref{homo} in particular follows from this result.
However, we prove this here as our discussion focuses using the structure of Damek-Ricci spaces. 


\item[(3)]
The system of partial differential equations \eqref{(1,1)'}--\eqref{(4,3)'} exhibits the following structural properties.
First, consider the change of variable $e^a = w$, and define new functions $F_3 = w f_3$ and $F_4 = w f_4$.  
Then, equations \eqref{(3,3)'} and \eqref{(4,3)'} can be rewritten in terms of the variables $(z, w) \in \mathbb{R} \times \mathbb{R}_+$, and the pair $(F_3, F_4)$ satisfies the Cauchy--Riemann equations:
\[
\dfrac{\partial F_3}{\partial z}
= \dfrac{\partial F_4}{\partial w}, \quad
\dfrac{\partial F_3}{\partial w}
= -\dfrac{\partial F_4}{\partial z}.
\]


\item[(4)]
Second, define the differential operators
\[
D_x := \dfrac{\partial}{\partial x} - \dfrac{y}{2} \dfrac{\partial}{\partial z}, \quad
D_y := \dfrac{\partial}{\partial y} + \dfrac{x}{2} \dfrac{\partial}{\partial z}.
\]
Then equations \eqref{(1,1)'} and \eqref{(2,2)'} yield $D_x f_1 = D_y f_2$, while equation \eqref{(2,1)'} gives $D_y f_1 + D_x f_2 = 0$.
These imply that the pair $(f_1, f_2)$ satisfies a Cauchy--Riemann-type system with respect to $D_x$ and $D_y$.
\end{itemize}
\end{rem}

\vspace{0.1in}

Now we show that radial harmonic functions in the $zw$-plane give rise to Killing vector fields.

\begin{prop}\label{radial}
If $F_3$ is a radial harmonic function on the $zw-$plane, then $\xi$
is a Killing vector field on $\mathbb{C}H^2$.
\end{prop}
\begin{proof}
Now suppose that $F_3$ is a radial harmonic function, then
by the characterization of radial harmonic functions, 
we have that 
$$F_3(x, y, z,w) = \frac{C_0(x, y)}{2} \log(z^2+{w}^2) + C_1(x,y).$$
Because the CR equations are satisfied, we have that 
$$F_{4}(x, y, z, w) =  C_0(x, y) \arctan\left(\frac{w}{z}\right)+C_2(x,y). $$
{Here $C_1,C_2$ are arbitrary functions depending only on $(x,y)$}.
Consequently functions
\begin{align*}
f_3(x,y,z,a)=&e^{-a}\left(\frac{C_0(x, y)}{2}\log \left(e^{2 a}+z^2\right)
+C_1(x,y)\right),\\
f_4(x,y,z,a)=&e^{-a}\left(C_0(x, y) \arctan\left(\frac{e^a}{z}\right)+C_2(x,y)\right),
\end{align*}
satisfies \eqref{(3,3)'} and \eqref{(4,3)'} and $C_1$ and $C_2$ are functions of the variables $x$ and $y$.
However, since the function $\arctan\left(\dfrac{e^a}{z}\right)$ is not defined as $z \to 0$, it follows that $C_0$ must vanish.
Hence, 
\begin{align*}
f_3(x,y,z,a)=&e^{-a} C_1(x,y),\\
f_4(x,y,z,a)=&e^{-a} C_2(x,y).
\end{align*}
From equation~\eqref{(4,1)'}, we obtain
\begin{equation*}
\dfrac{\partial}{\partial a}\left(e^{a/2}f_1\right) =
-\dfrac{\partial C_2}{\partial x}.
\end{equation*}
Integrating both sides with respect to $a$, we find that
\begin{equation}\label{f1}
e^{a/2}f_1(x,y,z,a) = -a
\dfrac{\partial C_2}{\partial x} + \phi(x,y,z),
\end{equation}
where $\phi$ is some function of $x,y$ and $z$.
If $a=0$, then the aforementioned equation yields that
$$f_1(x,y,z,0) =\phi(x,y,z).$$
Consequently, (\ref{f1}) yields that
\begin{equation}\label{f1'}
e^{a/2}f_1(x,y,z,a) = -a
\dfrac{\partial C_2}{\partial x} + f_1(x,y,z,0).
\end{equation}
Now differentiating $(\ref{f1'})$ with respect to $x$
and $z$ we obtain that LHS of $(\ref{(1,1)'})$ as:
\begin{equation}\label{f2}
e^{-a/2}\left(-a \dfrac{\partial^2C_2}{\partial x^2}
 +  \dfrac{\partial f_1}{\partial x}(x,y,z,0) -\dfrac{y}{2}\dfrac{\partial f_1}{\partial z}(x,y,z,0) \right).
\end{equation}
on the other hand RHS of $(\ref{(1,1)'})$ is:
\begin{equation}\label{f3}
- \dfrac{e^{-3a/2}}{2} C_2(x,y). 
\end{equation}
Comparing (\ref{f2}) and (\ref{f3}), and dividing both sides by $a$, we obtain that 
as $a \rightarrow \infty$, 
$\dfrac{\partial^2C_2}{\partial x^2} = 0.$ 
Again using (\ref{f2}) we have,
\begin{equation*}
\dfrac{\partial f_1}{\partial x}(x,y,z,0) -\dfrac{y}{2}\dfrac{\partial f_1}{\partial z}(x,y,z,0)  = - \dfrac{e^{-a}}{2} C_2(x,y). 
\end{equation*}
Consequently, it follows that 
$C_2(x,y) = 0$ and hence, Cauchy-Riemann equations 
imply that $C_1(x,y) = 0$ as well.
\end{proof}

\begin{rem}
It can be proved that any Killing vector field on a simply connected, complete, non-compact, non-flat harmonic manifold can not be of constant length or a closed Killing vector field.
In particular, this holds for $\mathbb{C}H^2$.
If $\xi$ is a Killing vector field, then
$\Delta h =  ||\nabla  \xi||^2  -  \mathrm{Ric}(\xi,\xi)$, where $h  :=  \frac{1}{2} g(\xi,\xi)  =  \frac{1}{2} ||\xi||^2$.
But if $\xi$ is of constant length, then it implies that  $\xi$ is a parallel vector field. 
This implies that the harmonic manifold is reducible, hence flat. 
Also, by \cite{NGSB}, it follows that it can't be a closed  vector field.
\end{rem}



Now, for a more general result, we study harmonic functions on the plane.

\section{Triviality of conformal vector fields on $\mathbb{C}H^2$}

In this section, we prove that any conformal vector field on the complex hyperbolic plane $\mathbb{C}H^2$ must, in fact, be Killing.
  
We aim to show that $f_4$ is $a$-independent, which immediately implies, via \eqref{(4,4)}, that the potential $\rho$ of the conformal vector field $\xi$ vanishes,
so $\xi$ becomes a Killing field.
In Section \ref{pre}, we introduce a harmonic polynomial expansion for $F_4$.
Section \ref{mrp} then states our main result (Theorem \ref{mainresult}) and its proof,
culminating in the corollary that $\xi$ is Killing.

\subsection{Series representation of $f_4$ via harmonic polynomials}
\label{pre}

The functions $f_3$ and $f_4$ are originally defined as functions of four variables $(x,y,z,a)$, and satisfy a system of partial differential equations.  
As observed above, by introducing the new variables $F_3 = e^a f_3$ and $F_4 = e^a f_4$, and restricting attention to $(z, w)$ with $w = e^a$, these functions become harmonic with respect to $z$ and $w$.
Importantly, however, they remain parametrized by $(x, y)$.  
This allows us to consider $F_3$ and $F_4$ as families of harmonic functions on the $(z, w)$-plane, depending smoothly on the parameters $(x, y)$.

Since each of these functions is harmonic in $(z, w)$, it follows that they admit expansions in terms of harmonic polynomials with coefficients depending on $(x, y)$.
More precisely, by the well-known completeness of harmonic polynomials, any harmonic function defined in a domain can be uniquely expanded in an infinite series of harmonic polynomials, which converges uniformly on compact subsets.
Therefore, we adopt this representation to analyze $F_3$ and $F_4$,
as it allows us to systematically examine their structure while avoiding unnecessary complexity from non-harmonic terms such as logarithmic components.
Moreover, Proposition 5.5 in \cite{ABR} ensures that any polynomial function can be decomposed as the sum of a harmonic polynomial and a multiple of the squared Euclidean norm, $x_1^2 + \cdots + x_n^2$.
In particular, by restricting our attention to expansions in harmonic polynomials, we can study the structure of $F_3$ and $F_4$ more transparently.
For instance, in the case of $n=2$, Theorem 5.25 in \cite{ABR} states that the basis for harmonic polynomials consists of powers of $\zeta=z+iw$ and $\bar{\zeta}$.
This approach enables us to isolate the core aspects of the problem, avoiding potential difficulties arising from multi-valued or singular components.
The subsequent sections therefore proceed under this framework, focusing on the implications and properties of these harmonic expansions.

We set functions $F_3^{[m]}$ and $F_4^{[m]}$ as
$$
(C_1^{[m]} + i C_2^{[m]}) (z + i w)^m=F_3^{[m]} + i F_4^{[m]},
$$
i.e., 
\begin{align*}
F_3^{[m]}= &\Re\left[ (C_1^{[m]} + i C_2^{[m]}) (z + i w)^m\right],\\
F_4^{[m]}= &\Im\left[ (C_1^{[m]} + i C_2^{[m]}) (z + i w)^m\right],
\end{align*}
where $C_i^{[m]}=C_i^{[m]}(x,y)$, $i=1, 2$, are smooth functions defined on the $(x,y)$-plane $\mathbb{R}^2$.
This decomposition allows us to systematically analyze the real and imaginary parts of the function through harmonic polynomial expansions.
A straightforward manipulation then yields
\begin{align}
F_3^{[m]}
= & \frac{1}{2} \left\{ (C_1^{[m]} + i C_2^{[m]}) (z + i w)^m + (C_1^{[m]} - i C_2^{[m]}) (z - i w)^m\right\},\notag\\
=& \frac{1}{2} \left\{ C_1^{[m]}((z + i w)^m+(z - i w)^m)
+i C_2^{[m]}((z + i w)^m-(z - iw)^m)\right\}.\label{f3_o}
\end{align}
Next, we invoke the binomial expansion of $(z \pm i w)^m$,
which can be written as
\begin{equation*}
(z \pm i w)^m
= \sum_{k=0}^{m} \binom{m}{k} z^{m-k} (\pm i w)^k
= \sum_{k=0}^{m} \binom{m}{k} z^{m-k}\,i^{\pm k}\, w^k,
\end{equation*}
where we have used
\begin{equation*}
(\pm i)^k=\left(i^{\pm 1}\right)^k=i^{\pm k}
\end{equation*}
up to signs that align with the binomial terms.
Each term in this expansion corresponds to a specific harmonic component in the decomposition.
To handle the powers of $i$,
we further utilize the standard form of Euler’s formula,
$$
i^\theta=\cos \frac{\pi}{2}\theta + i \sin \frac{\pi}{2}\theta.
$$
Combining these expressions reveals that all relevant real and imaginary contributions stem from the trigonometric parts $\cos\left(\frac{\pi}{2}k\right)$ and $\sin\left(\frac{\pi}{2}k\right)$.
Consequently, upon collecting real terms, one obtains
\begin{equation*}
F_3^{[m]}=\sum_{k=0}^{m} \binom{m}{k} z^{m-k} w^k \left(C_1^{[m]} \cos \frac{\pi}{2}k - C_2^{[m]} \sin \frac{\pi}{2}k \right).
\end{equation*}
Hence, $F_3^{[m]}$ can be expressed as a finite sum of monomials in $z$ and $w$,
each weighted by trigonometric coefficients determined by $\cos\left(\frac{\pi}{2}k\right)$ and $\sin\left(\frac{\pi}{2}k\right)$.
This decomposition makes explicit the real-part structure inherent in $(z+iw)^m$ under the given linear combination involving $C_1^{[m]}$ and $C_2^{[m]}$.
Likewise, $F_4^{[m]}$ takes the analogous form given by
\begin{equation*}
F_4^{[m]}
=\sum_{k=0}^{m} \binom{m}{k} z^{m-k} w^k \left(C_1^{[m]} \sin \frac{\pi}{2}k + C_2^{[m]} \cos \frac{\pi}{2}k \right).
\end{equation*}

Using $F_{3}^{[m]}$ and $F_{4}^{[m]}$,
the functions $f_{3}$ and $f_{4}$ can be represented by the following infinite series.
\begin{align}
f_3= &e^{-a}\left(\sum_{m=1}^{\infty}F_{3}^{[m]}+ C_3\right)\notag\\
=& e^{-a}\left\{\sum_{m=1}^{\infty}\left(\sum_{k=0}^{m} \binom{m}{k} z^{m-k} e^{k a} \left(C_1^{[m]} \cos \frac{\pi}{2}k - C_2^{[m]} \sin \frac{\pi}{2}k \right)\right) + C_3\right\},\label{f3simple}\\
f_4 = & e^{-a}\left(\sum_{m=1}^\infty F_4^{[m]}+ C_4\right)\notag\\
= & e^{-a}\left\{\sum_{m=1}^\infty\left(\sum_{k=0}^{m} \binom{m}{k} z^{m-k} e^{k a} \left(C_1^{[m]} \sin \frac{\pi}{2}k + C_2^{[m]} \cos \frac{\pi}{2}k \right)\right) + C_4\right\},\label{f4simple}
\end{align}
where $C_3=C_3(x,y)$ and $C_4=C_4(x,y)$ are some functions on $(x,y)$-plane $\mathbb{R}^2$.

\begin{rem}
It should be noted that (\ref{f3simple}) and (\ref{f4simple}) they represent power series in the neighbourhood of zero, consequently the solution pair $f_3$ and $f_4$ is in a neighbourhood of zero.
The solution in any neighbourhood around  some point $z_0 + i w_0 $ is given by replacing $z + i w$ by $(z-z_0) + i (w - w_0)$. Hence, the above computations remains unaltered.
\end{rem}

\subsection{Main result and its proof}
\label{mrp}

We now state our main result.

\begin{thm}\label{mainresult}
Under the assumptions given above, $C_{4} = 0$.
Moreover, for every positive integer $m$,
$C_{1}^{[m]} = 0$, and for all positive integer $m \ge 3$, we have $C_{2}^{[m]} = 0$.
\end{thm}

\begin{proof}
To simplify the notation in what follows, we define
\[
\mathcal{C}^{(p)}_{[m,k]}
\;:=\;
\dfrac{\partial^p C_1^{[m]}}{\partial x^p}\sin\dfrac{\pi}{2} k
+
\dfrac{\partial^p C_2^{[m]}}{\partial x^p}\cos\dfrac{\pi}{2} k,
\]
for integers $m \ge 1$, $k \ge 0$, and $p = 0, 1, 2$.

From equations \eqref{(4,1)'} and \eqref{(4,2)'}, $f_1$ and $f_2$ can be determined in terms of $f_4$.
For instance, $f_1$ can be expressed by the following integral formula:
\begin{equation}\label{f1_0}
e^{a/2}f_1=-\int e^a\left(\dfrac{\partial f_4}{\partial x}-\dfrac{y}{2}\dfrac{\partial f_4}{\partial z}\right)\,da+C_5(x,y,z),
\end{equation}
where $C_5(x,y,z)$ is an integration “constant” that may depend on $(x,y,z)$ but not on $a$.
Equation \eqref{f1_0} could be integrated further to obtain an explicit expression for $f_1$.
However, we will not pursue that computation here; instead, we proceed with the proof of our main result.

From \eqref{(1,1)'}, we have
\begin{align}\label{16}
0=& - \dfrac{\partial}{\partial x}\left(e^{a/2} f_1\right)
+\dfrac{y}{2}\dfrac{\partial}{\partial z}\left(e^{a/2} f_1\right)
+e^{-a/2}\dfrac{\partial}{\partial a}\left(e^{a/2} f_4\right).
\end{align}
The first and second terms on the right-hand side of \eqref{16} become
$$
\int e^a\left(
\dfrac{\partial^2 f_4}{\partial x^2}
-y\dfrac{\partial^2 f_4}{\partial x \partial z}
+\dfrac{y^2}{4}\dfrac{\partial^2 f_4}{\partial z^2}
\right) da
-\dfrac{\partial C_5}{\partial x}
+\dfrac{y}{2} \dfrac{\partial C_5}{\partial z}.
$$
By computing each term of the above equation, we obtain the following:
\begin{align*}
\int e^{a}\dfrac{\partial^2 f_4}{\partial x^2} da
=&\int \left\{\sum_{m=1}^\infty\left(\sum_{k=0}^{m} \binom{m}{k} z^{m-k} e^{k a} \mathcal{C}^{(2)}_{[m,k]} + \dfrac{\partial^2C_4}{\partial x^2}\right)\right\} da\\
=&\sum_{m=1}^\infty\left(\sum_{k=1}^{m} \binom{m}{k}\dfrac{1}{k} z^{m-k} e^{k a} \mathcal{C}^{(2)}_{[m,k]}\right) + a \left(
\sum_{m=1}^\infty z^m \dfrac{\partial^2 C^{[m]}_2}{\partial x^2}
+ \dfrac{\partial^2C_4}{\partial x^2}
\right).
\end{align*}

Now, after some computations and simplifications, we affirm
\begin{align*}
\int e^{a}\dfrac{\partial^2 f_4}{\partial x\partial z} da
=&\dfrac{\partial}{\partial z}\int \left\{\sum_{m=1}^\infty\left(\sum_{k=0}^{m} \binom{m}{k} z^{m-k} e^{k a} \mathcal{C}^{(1)}_{[m,k]}\right) + \dfrac{\partial C_4}{\partial x}\right\} da\\
=&\sum_{m=1}^\infty\left(\sum_{k=1}^{m} \binom{m+1}{k}\dfrac{m-k+1}{k} z^{m-k} e^{k a} \mathcal{C}^{(1)}_{[m+1,k]}\right)
+a \sum_{m=1}^\infty m z^{m-1} \dfrac{\partial C_2^{[m]}}{\partial x}.
\end{align*}
And,
\begin{align*}
\int e^{a}\dfrac{\partial^2 f_4}{\partial z^2} da
=&\dfrac{\partial^2}{\partial z^2}
\int
\left\{\sum_{m=1}^\infty\left(\sum_{k=0}^{m} \binom{m}{k} z^{m-k} e^{k a} \mathcal{C}^{(0)}_{[m,k]}\right)\right\}
da\\
=& a \sum_{m=1}^\infty (m+1)m z^{m-1} C_2^{[m+1]}\\
&+\sum_{m=1}^\infty\left(\sum_{k=1}^{m} \binom{m+2}{k} \dfrac{(m-k+2)(m-k+1)}{k}z^{m-k} e^{k a} \mathcal{C}^{(0)}_{[m+2,k]}\right).
\end{align*}
Furthermore, the third term of \eqref{16} is given by the following:
\begin{align*}
e^{-a/2}\dfrac{\partial}{\partial a}\left(e^{a/2} f_4\right)
=&\dfrac{1}{2}f_4+\dfrac{\partial f_4}{\partial a}\\
=& \sum_{m=1}^\infty\left(\sum_{k=0}^{m} \binom{m}{k} \left(k-\dfrac{1}{2}\right)z^{m-k} e^{(k - 1)a} \mathcal{C}^{(0)}_{[m,k]}\right)-\dfrac{1}{2}e^{-a} C_4\\
=&\sum_{m=1}^\infty\left(\sum_{k=1}^{m} \binom{m+1}{k+1} \left(k+\dfrac{1}{2}\right)z^{m-k} e^{k a} \mathcal{C}^{(0)}_{[m+1,k+1]}\right)\\
&+\dfrac{1}{2}\sum_{m=1}^\infty m z^{m-1} C^{[m]}_1
-\dfrac{1}{2} e^{-a}\sum_{m=1}^\infty z^m C^{[m]}_2
-\dfrac{1}{2}e^{-a} C_4.
\end{align*}
Thus, equation \eqref{16} can be rewritten as follows:
\begin{align*}
0=&
\sum_{m=1}^\infty\left(\sum_{k=1}^{m} \binom{m}{k}\dfrac{1}{k} z^{m-k} e^{k a} \mathcal{C}^{(2)}_{[m,k]}\right)
+ a \left(
\sum_{m=1}^\infty z^m \dfrac{\partial^2 C^{[m]}_2}{\partial x^2}
+ \dfrac{\partial^2C_4}{\partial x^2}
\right)\\
&-y\sum_{m=1}^\infty\left(\sum_{k=1}^{m} \binom{m+1}{k}\dfrac{m-k+1}{k} z^{m-k} e^{k a} \mathcal{C}^{(1)}_{[m+1,k]}\right)\\
&-ay \sum_{m=1}^\infty m z^{m-1} \dfrac{\partial C_2^{[m]}}{\partial x}
+\dfrac{ay^2}{4} \sum_{m=1}^\infty (m+1)m z^{m-1} C_2^{[m+1]}\\
&+ \dfrac{y^2}{4}\sum_{m=1}^\infty\left(\sum_{k=1}^{m} \binom{m+2}{k} \dfrac{(m-k+2)(m-k+1)}{k}z^{m-k} e^{k a} \mathcal{C}^{(0)}_{[m+2,k]}\right)
-\dfrac{\partial C_5}{\partial x}
+\dfrac{y}{2} \dfrac{\partial C_5}{\partial z}\\
&+\sum_{m=1}^\infty\left(\sum_{k=1}^{m} \binom{m+1}{k+1} \left(k+\dfrac{1}{2}\right)z^{m-k} e^{k a} \mathcal{C}^{(0)}_{[m+1,k+1]}\right)\\
&+\dfrac{1}{2}\sum_{m=1}^\infty m z^{m-1} C^{[m]}_1
-\dfrac{1}{2} e^{-a}\sum_{m=1}^\infty z^m C^{[m]}_2
-\dfrac{1}{2}e^{-a} C_4.
\end{align*}
That is,
\begin{align*}
0=
&\sum_{m=1}^\infty \sum_{k=1}^{m} z^{m-k} e^{k a}
\left\{
\binom{m}{k}\dfrac{1}{k} \mathcal{C}^{(2)}_{[m,k]}
-y \binom{m+1}{k}\dfrac{m-k+1}{k} \mathcal{C}^{(1)}_{[m+1,k]}\right.\\
&\left.+ \dfrac{y^2}{4} \binom{m+2}{k} \dfrac{(m-k+2)(m-k+1)}{k} \mathcal{C}^{(0)}_{[m+2,k]}
+ \binom{m+1}{k+1} \left(k+\dfrac{1}{2}\right) \mathcal{C}^{(0)}_{[m+1,k+1]}\right\}\\
&+ a \left\{
\sum_{m=1}^\infty z^m \dfrac{\partial^2 C^{[m]}_2}{\partial x^2}
+ \dfrac{\partial^2C_4}{\partial x^2}
+\sum_{m=1}^\infty m z^{m-1}\left(-y \dfrac{\partial C_2^{[m]}}{\partial x}
+ \dfrac{m+1}{4} y^2 C_2^{[m+1]}\right)
\right\}\\
&-\dfrac{1}{2}e^{-a}\left(
\sum_{m=1}^\infty z^m C^{[m]}_2+C_4
\right)
+\dfrac{1}{2}\sum_{m=1}^\infty m z^{m-1} C^{[m]}_1
-\dfrac{\partial C_5}{\partial x}
+\dfrac{y}{2} \dfrac{\partial C_5}{\partial z}.
\end{align*}
The above equation holds for any $a$.
Since the exponential functions $\{e^{ka}\}_{k\in\mathbb{Z}}$ are linearly independent over $\mathbb{R}$,
it follows that each block corresponding to a fixed $k$ must vanish separately.
This yields
\begin{multline}\label{16mk}
\sum_{m=k}^\infty z^{m-k}
\left\{
\binom{m}{k}\dfrac{1}{k} \mathcal{C}^{(2)}_{[m,k]}
-y \binom{m+1}{k}\dfrac{m-k+1}{k} \mathcal{C}^{(1)}_{[m+1,k]}\right.\\
+ \dfrac{y^2}{4} \binom{m+2}{k} \dfrac{(m-k+2)(m-k+1)}{k} \mathcal{C}^{(0)}_{[m+2,k]}\\
\left.+ \binom{m+1}{k+1} \left(k+\dfrac{1}{2}\right) \mathcal{C}^{(0)}_{[m+1,k+1]}\right\}=0,
\end{multline}
for any positive integer $k$ and
\begin{equation}
\label{16-1}
\sum_{m=1}^\infty \left\{z^m \dfrac{\partial^2 C^{[m]}_2}{\partial x^2}
+ m z^{m-1}\left(-y \dfrac{\partial C_2^{[m]}}{\partial x}
+ \dfrac{m+1}{4} y^2 C_2^{[m+1]}\right)\right\}+ \dfrac{\partial^2C_4}{\partial x^2}=0,
\end{equation}
\begin{equation}
\label{16-2}
\sum_{m=1}^\infty z^m C^{[m]}_2+C_4=0,
\end{equation}
\begin{equation}
\label{16-3}
\dfrac{1}{2}\sum_{m=1}^\infty m z^{m-1} C^{[m]}_1
-\dfrac{\partial C_5}{\partial x}
+\dfrac{y}{2} \dfrac{\partial C_5}{\partial z}=0.
\end{equation}
Since the equation \eqref{16-2} holds for any $z$,
we obtain
$$
C_4(x,y)=0,\qquad
C_2^{[m]}(x,y)=0\quad( m=1,2,\ldots).
$$
It is immediately clear that these satisfy equation \eqref{16-1} without contradiction.
On the other hand,
since equation \eqref{16mk} holds for any $z$,
it follows that for any positive integer $k$ and for any $m\ge k$,
\begin{multline*}
\binom{m}{k}\dfrac{1}{k} \mathcal{C}^{(2)}_{[m,k]}
-y \binom{m+1}{k}\dfrac{m-k+1}{k} \mathcal{C}^{(1)}_{[m+1,k]}\\
+ \dfrac{y^2}{4} \binom{m+2}{k} \dfrac{(m-k+2)(m-k+1)}{k} \mathcal{C}^{(0)}_{[m+2,k]}\\
+ \binom{m+1}{k+1} \left(k+\dfrac{1}{2}\right) \mathcal{C}^{(0)}_{[m+1,k+1]}=0,
\end{multline*}
holds.
For $k=2$, we obtain
\begin{multline*}
-\dfrac{m-1}{4}\dfrac{\partial^2 C_2^{[m]}}{\partial x^2}
+\dfrac{(m+1)}{4} y \dfrac{\partial C_2^{[m+1]}}{\partial x}
-\dfrac{(m+2)(m+1)(m-1)}{16} y^2 C_2^{[m+2]}\\
-(m+1)(m-1)\dfrac{5}{12}C_1^{[m+1]}=0,
\end{multline*}
for any $m\ge 2$.
Given $C_2^{[m]}(x,y)=0$ for any positive integer $m$,
it follows that $C_1^{[m]}(x,y)=0$ for $m\ge 3$.
\end{proof}

\begin{note}
Since $F_3$ and $F_4$ are harmonic in $(z,w)$, they are real-analytic.
Hence, in a neighborhood of the origin they admit convergent expansions into homogeneous harmonic polynomials, and these expansions together with their finite $(z,w)$-derivatives converge uniformly.
Therefore termwise differentiation and substitution into the system of equations are justified, and the coefficient comparisons used above are valid.
\end{note}

\begin{rem}
For $k=1$, we obtain
\begin{multline}\label{16mk=1}
\dfrac{\partial^2 C_1^{[m]}}{\partial x^2}
-(m+1) y \dfrac{\partial C_1^{[m+1]}}{\partial x}
+\dfrac{(m+2)(m+1)}{4} y^2 C_1^{[m+2]}\\
-\dfrac{3(m+1)}{4}C_2^{[m+1]}=0,
\end{multline}
for any $m\ge 1$.
From the previous results, equation \eqref{16mk=1} is meaningful only for $m = 1$ and $m = 2$.
In fact, for $m = 3$, all terms in \eqref{16mk=1} vanish.
For $m=1,2$, equation  \eqref{16mk=1} takes the forms:
\begin{equation}\label{C1[1][2]-1}
\dfrac{\partial^2 C_1^{[1]}}{\partial x^2}
-2 y \dfrac{\partial C_1^{[2]}}{\partial x}=0,\qquad
\dfrac{\partial^2 C_1^{[2]}}{\partial x^2}=0.
\end{equation}
Recall that this result was derived from equations \eqref{(1,1)'} and \eqref{(4,1)'}.
A similar argument can be applied to equations (18) and \eqref{(4,2)'},
leading to the following result, which is analogous to equation \eqref{C1[1][2]-1}:
\begin{equation*}\label{C1[1][2]-2}
\dfrac{\partial^2 C_1^{[1]}}{\partial y^2}
+2 x \dfrac{\partial C_1^{[2]}}{\partial y}=0,\qquad
\dfrac{\partial^2 C_1^{[2]}}{\partial y^2}=0,
\end{equation*}
which implies that $C_1^{[1]}$ and $C_1^{[2]}$ are at most first-degree polynomials in $x$ and $y$ (in fact, they are either constant functions or first-degree polynomials).
\end{rem}

\begin{cor}
The potential function $f$ of the conformal vector field $\xi$ is zero.
That is, $\xi$ is a Killing vector field.
\end{cor}

\begin{proof}
Based on the above considerations, let us explicitly express $f_4$.
From \eqref{f4simple}, we have
\begin{align*}
f_4 = & e^{-a}\left(\sum_{m=1}^2\left(\sum_{k=0}^{m} \binom{m}{k} z^{m-k} e^{k a} C_1^{[m]} \sin \frac{\pi}{2}k\right)\right)\\
=&e^{-a}\left(
\sum_{k=0}^{1} \binom{1}{k} z^{1-k} e^{k a} C_1^{[1]} \sin \frac{\pi}{2}k
+
\sum_{k=0}^{2} \binom{2}{k} z^{2-k} e^{k a} C_1^{[2]} \sin \frac{\pi}{2}k
\right)\\
=&e^{-a}\left(
e^{a} C_1^{[1]}
+
2 z e^{a} C_1^{[2]}
\right)\\
=&C_1^{[1]}(x,y)+2 z C_1^{[2]}(x,y),
\end{align*}
which shows that $f_4$ does not depend on the variable $a$,
implying that from \eqref{(4,4)} the potential function $\rho$ of $\xi$ is zero.
\end{proof}

\section{On a complex hyperbolic space $\mathbb{C}H^n$}

We now extend the arguments of Section \ref{CH2} from $\mathbb{C}H^2$ to the complex hyperbolic space $\mathbb{C}H^n$, $n>2$.
Similar to the $n=2$ case, 
a complex hyperbolic space $\mathbb{C}H^n$ is also the Damek-Ricci space $S=N\ltimes \mathbb{R}A$.
In the general case of $\mathbb{C}H^n$,
the Lie algebra $\mathfrak{n}$ decomposes as $\mathfrak{v}\oplus \mathfrak{z}$ with $\dim\mathfrak{z}=1$ and $\dim \mathfrak{v}=2(n-1)$.
We label an orthonormal basis $\{ V_i, J_Z V_i\}_{i=1\ldots,n-1}$ for $\mathfrak{v}$ and let $Z$ generate $\mathfrak{z}$,
where $V\in \mathfrak{v}$ and $Z\in\mathfrak{z}$ are unit vectors and satisfy
$$
[V_i, V_j]=[J_Z V_i, J_ZV_j]=0,
\quad\mbox{and}\quad
[V_i, J_Z V_j]=\delta_{ij} Z.
$$
Considering the elements of the Lie algebra as left-invariant vector fields on $S$,
the components of the Lie derivative of the metric $g$ are as follows:
\begin{equation*}
\mathcal{L}_{V_i}g(V_j,A)=\frac{1}{2}\delta_{ij},\quad
\mathcal{L}_{V_i}g(J_ZV_j,Z)=-\delta_{ij},
\end{equation*}
\begin{equation*}
\mathcal{L}_{J_ZV_i}g(V_j,Z)=\delta_{ij},\quad
\mathcal{L}_{J_ZV_i}g(J_ZV_j,A)=\dfrac{1}{2}\delta_{ij},
\end{equation*}
\begin{equation*}
\mathcal{L}_{Z}g(Z,A)=1,
\end{equation*}
\begin{equation*}
\mathcal{L}_{A}g(V_i,V_j)=\mathcal{L}_{A}g(J_ZV_i,J_ZV_j)=-\delta_{ij},\quad
\mathcal{L}_{A}g(Z,Z)=-2.
\end{equation*}
Here, $1\le i, j\le n-1$ and all components not listed above are zero.
When the matrix $\mathcal{L}_{V_i}g$ is partitioned into $2\times 2$ blocks,
the $(n, i)$-th block becomes the matrix $\left(
\begin{array}{cc}
0&-1\\
\frac{1}{2}&0
\end{array}
\right)$, and the $(i, n)$-th block becomes its transpose.
Furthermore, all other blocks are zero matrices:
\begin{align*}
\mathcal{L}_{V_i}g
=&\left(
\begin{array}{cc:c:cc:c:cc}
\hspace{15pt}&&\hspace{5pt}&&&\hspace{5pt}&&\\
&&&&&&&\\
\hdashline
&&&&&&&\\
\hdashline
&&&&&&0&\frac{1}{2}\\
&&&&&&-1&0\\
\hdashline
&&&&&&&\\
\hdashline
&&&0&-1&&&\\
&&&\frac{1}{2}&0&&&\\
\end{array}
\right).
\end{align*}
Similarly, when the matrix $\mathcal{L}_{J_ZV_i}g$ is partitioned into blocks, the $(n, i)$-th block and the $(i, n)$-th block become the matrix $\left(
\begin{array}{cc}
1&0\\
0&\frac{1}{2}
\end{array}
\right)$.
\begin{align*}
\mathcal{L}_{J_ZV_i}g
=&\left(
\begin{array}{cc:c:cc:c:cc}
\hspace{15pt}&&\hspace{5pt}&&&\hspace{5pt}&&\\
&&&&&&&\\
\hdashline
&&&&&&&\\
\hdashline
&&&&&&1&0\\
&&&&&&0&\frac{1}{2}\\
\hdashline
&&&&&&&\\
\hdashline
&&&1&0&&&\\
&&&0&\frac{1}{2}&&&\\
\end{array}
\right).
\end{align*}
The matrices $\mathcal{L}_Zg$ and $\mathcal{L}_Ag$ are given by
\begin{align*}
\mathcal{L}_Zg
=&\left(
\begin{array}{ccc:cc}
0&&&&\\
&\ddots&&&\\
&&0&&\\
\hdashline
&&&0&1\\
&&&1&0
\end{array}
\right),&
\mathcal{L}_Ag
=&\left(
\begin{array}{ccc:cc}
-1& & & & \\
 &\ddots& & & \\
 & &-1& & \\
 \hdashline
 & & &-2&0\\
 & & &0&0 
\end{array}
\right).
\end{align*}
As in the case of $\mathbb{C}H^2$,
with respect to the natural coordinate system $(x^1, y^1, \ldots, x^{n-1}, y^{n-1}, z, a)$ determined by the generators of the Lie algebra,
we obtain
$$
V_i=e^{a/2}\left(\dfrac{\partial}{\partial x^i}
-\dfrac{y^i}{2}\dfrac{\partial}{\partial z}\right),\ 
J_ZV_i=e^{a/2}\left(\dfrac{\partial}{\partial y^i}
+\dfrac{x^i}{2}\dfrac{\partial}{\partial z}\right),\ 
Z=e^{a}\dfrac{\partial}{\partial z},\ 
A=\dfrac{\partial}{\partial a},
$$
where $1\le i\le n-1$.
The condition that the vector field 
\begin{equation*}\label{xi1}
\xi=\sum_{i=1}^{n-1}\left(f_{1i}\,V_i+f_{2i}\,J_ZV_i\right)+f_3\,Z+f_4\,A
\end{equation*}
is a conformal vector field is equivalent to the following system of partial differential equations being satisfied:
\begin{align}
\label{(1,1)-i}
2e^{a/2}\left(\frac{\partial f_{1i}}{\partial x^i}-\frac{y^i}{2}\frac{\partial f_{1i}}{\partial z}\right)
-f_4
=&2\rho,\\
\label{(2,1)-i}
e^{a/2}\left(\frac{\partial f_{1i}}{\partial y^i}+\frac{x^i}{2}\frac{\partial f_{1i}}{\partial z}\right)
+
e^{a/2}\left(\frac{\partial f_{2i}}{\partial x^i}-\frac{y^i}{2}\frac{\partial f_{2i}}{\partial z}\right)
=&0,\\
\label{(2,2)-i}
2e^{a/2}\left(\frac{\partial f_{2i}}{\partial y}+\frac{x^i}{2}\frac{\partial f_{2i}}{\partial z}\right)
-f_4
=&2\rho,\\
\label{(3,1)-i}
e^{a} \frac{\partial f_{1i}}{\partial z}
+
e^{a/2}\left(\frac{\partial f_3}{\partial x^i}-\frac{y^i}{2}\frac{\partial f_3}{\partial z}\right)
+f_{2i}
=&0,\\
\label{(3,2)-i}
e^{a} \frac{\partial f_2}{\partial z}
+
e^{a/2}\left(\frac{\partial f_3}{\partial y^i}+\frac{x^i}{2}\frac{\partial f_3}{\partial z}\right)
-f_{1i}
=&0,\\
\label{(3,3)-i}
2e^{a} \frac{\partial f_3}{\partial z}
-2f_4
=&2\rho,\\
\label{(4,1)-i}
\frac{\partial f_{1i}}{\partial a}
+
e^{a/2}\left(\frac{\partial f_4}{\partial x^i}-\frac{y^i}{2}\frac{\partial f_4}{\partial z}\right)
+\dfrac{1}{2}f_{1i}
=&0,\\
\label{(4,2)-i}
\frac{\partial f_{2i}}{\partial a}
+
e^{a/2}\left(\frac{\partial f_4}{\partial y^i}+\frac{x^i}{2}\frac{\partial f_4}{\partial z}\right)
+\dfrac{1}{2}f_{2i}
=&0,\\
\label{(4,3)-i}
\frac{\partial f_3}{\partial a}
+
e^{a} \frac{\partial f_4}{\partial z}
+f_3
=&0,\\
\label{(4,4)-i}
2\frac{\partial f_4}{\partial a}
=&2\rho,
\end{align}
for $1\le i\le n-1$, and,
\begin{align}
\label{(i,j)-1}
e^{a/2}\left(\frac{\partial f_{1j}}{\partial x^i}
-\frac{y^i}{2}\frac{\partial f_{1j}}{\partial z}\right)
+
e^{a/2}\left(\frac{\partial f_{1i}}{\partial x^j}
-\frac{y^j}{2}\frac{\partial f_{1i}}{\partial z}\right)
=&0,\\
\label{(i,j)-2}
e^{a/2}\left(\frac{\partial f_{1j}}{\partial y^i}
+\frac{x^i}{2}\frac{\partial f_{1j}}{\partial z}\right)
+
e^{a/2}\left(\frac{\partial f_{2i}}{\partial y^j}
-\frac{x^j}{2}\frac{\partial f_{2i}}{\partial z}\right)
=&0,\\
\label{(i,j)-3}
e^{a/2}\left(\frac{\partial f_{1j}}{\partial y^i}
+\frac{x^i}{2}\frac{\partial f_{1j}}{\partial z}\right)
+
e^{a/2}\left(\frac{\partial f_{2i}}{\partial x^j}
-\frac{y^j}{2}\frac{\partial f_{2i}}{\partial z}\right)
=&0,\\
\label{(i,j)-4}
e^{a/2}\left(\frac{\partial f_{2j}}{\partial y^i}
+\frac{x^i}{2}\frac{\partial f_{2j}}{\partial z}\right)
+
e^{a/2}\left(\frac{\partial f_{2i}}{\partial y^j}
+\frac{x^j}{2}\frac{\partial f_{2i}}{\partial z}\right)
=&0,
\end{align}
for $1\le i\ne j\le n-1$.\\

It is important to note that the functions $f_{1i}, f_{2i} \; (1\le i\le n-1), f_3, f_4$ are functions of $2n$ variables;
however, by focusing on the four variables $(x^i, y^i, z, a)$,
we observe that equations \eqref{(1,1)-i} through \eqref{(4,4)-i} coincide exactly with the system of partial differential equations in the case of $\mathbb{C}H^2$.  
Therefore, by applying the same argument as in the proof of Theorem \ref{mainresult},
one can similarly show that $f_4$ is independent of the variable $a$,
and that the potential function $\rho$ vanishes.
From this, we obtain the following result:

\begin{thm}
There exist no nontrivial conformal vector fields on the complex hyperbolic space $\mathbb{C}H^n$ for any integer $n \geq 2$.
Hence, all conformal vector fields must be Killing. 
\end{thm}

\section*{Appendix}

Nickerson’s Remark~2 in \cite{Nickerson1985} concerns conformal changes of the metric: he observes that if a metric $h$ is locally symmetric and conformal to a locally symmetric metric $g$,
then $h$ must differ from $g$ only by a constant factor, provided the manifold is not conformally flat and has non-zero scalar curvature.
This argument can be reformulated in terms of local flows generated by conformal vector fields, rather than conformal changes of the metric itself.
Interpreted in this setting, Nickerson’s observation yields the following theorem, which we record here for convenience.

\medskip

\begin{thm}\label{Nickerson}
Let $(M^n,g)$ be a Riemannian manifold of dimension $n\ge4$ satisfying
(i) $(M,g)$ is locally symmetric, i.e.\ $\nabla R=0$;
(ii) $(M,g)$ is not conformally flat, i.e.\ $W\not\equiv0$; and
(iii) the scalar curvature $\mathrm{Scal}_g$ is nonzero (and hence constant, by local symmetry).
Then every conformal vector field $\xi$ on $(M,g)$ is Killing.
\end{thm}

\begin{proof}
Let $\xi$ be a conformal vector field, i.e. $\mathcal{L}_\xi g = 2\rho g$ for some smooth $\rho$. 
Let $\Phi : D\to M$ be the local flow of $\xi$, with $D\subset\mathbb{R}\times M$ open, and fix an arbitrary point $p\in M$.
Choose a neighborhood $U\ni p$ and $\varepsilon>0$ such that for all $t$ with $|t|<\varepsilon$ and $(t,q)\in D$ for every $q\in U$, the map
$$
\phi_t:=\Phi(t,\cdot): (U,g)\longrightarrow \bigl(\phi_t(U),g\bigr)
$$
is a conformal diffeomorphism onto its image. Set, on $U$,
$$
h \;:=\; \bigl(\phi_t|_U\bigr)^{*}g \;=\; e^{2\psi_t}\,g
$$
for some smooth function $\psi_t:U\to\mathbb{R}$ depending smoothly on $t$.
By definition of pullback, $\phi_t:(U,h)\to\bigl(\phi_t(U),g\bigr)$ is an isometry; hence $(U,h)$ is also locally symmetric. 
Since $W\not\equiv 0$ on $(U,g)$ and $n\ge 4$, the metrics $g$ and $h$ are two locally symmetric metrics on the same domain $U$ that are pointwise conformally related.

By Nickerson \cite[Remark~2]{Nickerson1985}, on a locally symmetric manifold with nonvanishing Weyl tensor there is no non-homothetic conformal relation between locally symmetric metrics. 
Therefore there exists a constant $c_t\in\mathbb{R}$ (depending on $t$ but constant on $U$) such that
$$
h \;=\; e^{2c_t}\,g \quad\text{on $U$}.
$$
Now $\mathrm{Scal}_h = e^{-2c_t}\,\mathrm{Scal}_g$ for a homothety, whereas $\mathrm{Scal}_h = \mathrm{Scal}_g\circ \phi_t$ for a pullback by $\phi_t$; since $\mathrm{Scal}_g$ is a nonzero constant by local symmetry, we must have $e^{-2c_t}=1$, hence $c_t=0$. 
Thus $(\phi_t|_U)^*g = g$ for all sufficiently small $t$, and differentiating at $t=0$ yields $\mathcal{L}_\xi g=0$ on $U$. 
As $p$ was arbitrary, $\mathcal{L}_\xi g=0$ on $M$, i.e. $\xi$ is Killing.
\end{proof}

\vspace{0.1in}

\begin{flushleft}
Hiroyasu Satoh\\
Liberal Arts and Sciences\\
Nippon Institute of Technology\\
4-1 Gakuendai, Miyashiro-machi, Saitama 345-8501, Japan.\\
E-mail: hiroyasu@nit.ac.jp
\end{flushleft}

\begin{flushleft}
Hemangi  Madhusudan Shah\\
	Harish-Chandra Research Institute\\ 
	A CI of Homi Bhabha National Institute\\ 
	Chhatnag Road, Jhunsi, Prayagraj-211019, India.\\	
	E-mail:	hemangimshah@hri.res.in
\end{flushleft}

\end{document}